\documentclass{amsart}

\usepackage{amssymb} 
\usepackage{graphicx}
\usepackage{amsmath}               
\usepackage{amsfonts}              
\usepackage{amsthm}                
\usepackage[utf8]{inputenc}
\usepackage{tikz-cd}
\usepackage{mathtools}
\usepackage{fancyhdr}

\newtheorem{twierdzenie}{Theorem} 
\newtheorem{lemat}[twierdzenie]{Lemma} 
\theoremstyle{definition}
\newtheorem{definicja}[twierdzenie]{Definition}
\newtheorem{obserwacja}[twierdzenie]{Observation}
\newtheorem{przyklad}[twierdzenie]{Example}
\newtheorem{wniosek}[twierdzenie]{Corollary}
\newtheorem{propozycja}[twierdzenie]{Proposition}

\title{A criterion for the existence of periodic points based on the eigenvalues of maps induced in cohomology}
\author{Michalina Horecka, Paweł Raźny}
\address{Institute of Mathematics \\
	Faculty of Mathematics and Computer Science \\
	Jagiellonian University in Cracow
	}
\email{Pawel.Razny@im.uj.edu.pl\\Michalina.Horecka@im.uj.edu.pl}


\begin{document}

\maketitle
\begin{abstract}
We present a criterion for the existence of periodic points based on the eigenvalues of maps induced in cohomology for spaces with rational cohomology isomorphic to a tensor product of a graded exterior algebra with generators in odd dimensions and a graded algebra with all elements of even degree. We give a number of natural examples of such spaces and provide some non-trivial ones. We also give a counterexample to a claim in \cite{duan} given there without proof.
\end{abstract}
\section{Introduction}
In this short paper we generalize the results from \cite{ls} and \cite{duan} provided for rational exterior spaces to spaces with rational cohomology isomorphic to a tensor product of a graded exterior algebra with generators in odd dimensions and a graded algebra with all elements of even degree. In doing so, we expand the methods presented in these paper to, for example manifolds arising as blow ups of $\mathbb CP^n$, all products of spheres and some bundles over spheres. By the K\"{u}nneth Theorem our results are automatically extended to products of such spaces. It has also came to our attention that in \cite{duan} the author claims without proof that a tower of odd dimensional sphere bundles is a rational exterior space. This claim has been quoted in several papers without verification. We provide two examples showing that it is in fact false.
\newline\indent We start the paper by recalling some of the results and methods used in the aforementioned articles. In the subsequent section we provide our main results. We follow this with some examples of non-trivial towers of sphere bundles to which our theorems apply (this can be found via some work with the Serre spectral sequence). We finish the paper with a section containing counterexamples to the claim in \cite{duan}.

\section{Preliminaries}

Let $X$ be a connected euclidean neighborhood retract, in short ENR, $f:X\to X$ a continuous map and $f^k:=\underbrace{f\circ\ldots\circ f}_{k}$ a  composition of $f$ with itself $k$ times (we use this convention also for the morphism of algebras). A mapping $f$ induces a graded ring homomorphism $H^*(f):H^*(X,\mathbb{Q})\to H^*(X,\mathbb{Q})$, $H^*(f)=H^0(f)\oplus H^1(f)\oplus\ldots$, where $H^n(f)$ is a map induced in the n-th cohomology group.\\
The number
$$
\mathcal{L}(f):=\sum_{n=0}^{\infty}(-1)^n\mathrm{tr}(H^n(f))
$$
is called the Lefschetz number of the map $f$. \\
The Lefschetz Fixed Point Theorem implies that if $\mathcal{L}(f)\neq 0$, then $f$ has a fixed point. Therefore, if $\mathcal{L}(f^k)\neq 0$ then $f$ has a periodic point of period $k$. However, we do not know what happens when $\mathcal{L}(f^k)=0$.\\ 
\indent Observe that we can define a Lefschetz number for any graded ring endomorphism $g:H^*(X,\mathbb{Q})\to H^*(X,\mathbb{Q})$.
The Lefschetz Fixed Point Theorem loses its meaning in this context, but later we will see that such an extended definition is useful.
\begin{definicja}
We say that a continuous map $f$ is Lefschetz periodic point free, we will call it LPPF, if $\mathcal{L}(f^k)=0$ for all $k\geq 1$.
\end{definicja}

\indent In this section we recall the theory given in \cite{duan} and reformulate it in language of ring endomorphisms of cohomology ring.\\

\indent We say that an element $a\in H^r(X,\mathbb{Q})$, for some $r>0$, is decomposable if one can find such pairs of cohomology classes
$$
(b_i,c_i)\in H^{p_i}(X,\mathbb{Q})\times H^{q_i}(X,\mathbb{Q}),
$$ 
where $p_i,q_i>0$ and $p_i+q_i=r$, that $a=\sum b_i\smile c_i$. \\
\indent The set of all such elements, lets call it $D^r(X,\mathbb{Q})$, is a subspace of $H^r(X,\mathbb{Q})$.
So the quotient $A^r(X,\mathbb{Q}):=H^r(X,\mathbb{Q})/D^r(X,\mathbb{Q})$ is a vector space over $\mathbb{Q}$. 
Now, for a graded ring homomorphism $f:H^*(X,\mathbb{Q})\to H^*(Y,\mathbb{Q})$ we have $f(D^r(X,\mathbb{Q}))\subset  D^r(Y,\mathbb{Q})$, so passing to the quotient we get a homomorphism $A^r(f):A^r(X,\mathbb{Q})\to A^r(Y,\mathbb{Q})$. Denote
$$
A(X,\mathbb{Q}):=A^0(X,\mathbb{Q})\oplus A^1(X,\mathbb{Q})\oplus\ldots
$$$$
A(f):=A^0(f)\oplus A^1(f)\oplus\ldots
$$
 
\begin{definicja}
Rational exterior space $X$ is an ENR for which we can find such homogeneous elements $x_{1},\ldots,x_{s}\in H^{\mathrm{odd}}(X,\mathbb{Q})$ that the inclusions $x_i\hookrightarrow H^*(X,\mathbb{Q})$ for $i\in\{1,\ldots,s\}$ give rise to a ring isomorphism $\Lambda_{\mathbb{Q}}(x_{1},\ldots,x_{s})\simeq H^*(X,\mathbb{Q})$ where $\Lambda_{\mathbb{Q}}(x_{1},\ldots,x_{s})$ is an exterior algebra on generators $x_{1},\ldots,x_{s}$.
\end{definicja}
\begin{przyklad}
Simplest examples of such spaces are products of odd dimensional spheres and Lie groups.
\end{przyklad}
We note that the methods used in \cite{duan} are algebraic in nature and hence the formula can be proven in the exact same fashion for a ring endomorphism of an exterior algebra. We state the following important result from \cite{duan} with this modification.
\begin{twierdzenie}\label{duan}
If  $f:H^*(X,\mathbb{Q})\to H^*(X,\mathbb{Q})$ is an endomorphism of cohomology ring of a rational exterior space $X$ and $\lambda_1,\ldots,\lambda_n$ are eigenvalues of $A(f)$, then $L(f^k)=\prod_{i=1}^n(1-\lambda_{i}^k).$
\end{twierdzenie}
\begin{wniosek}\label{cor niep}
If  $f:H^*(X,\mathbb{Q})\to H^*(X,\mathbb{Q})$ is an endomorphism of cohomology ring of a rational exterior space $X$ and $\lambda_1,\ldots,\lambda_n$ are eigenvalues of $A(f)$, then $L(f^k)=0$ for all $k\in\mathbb{N}$ if and only if there is $i_0\in\{1,\ldots,n\}$ such that $\lambda_{i_0}=1$.
\end{wniosek}
Lefschetz zeta function is defined as a formal series
$$
\zeta_{f}(t):=\mathrm{exp}\Bigg(\sum_{k=1}^{\infty}\frac{L(f^{k})}{k}t^{k}\Bigg).
$$
There is an equivalent form of the Lefschetz zeta function 
\begin{lemat}
The following equality holds
\begin{equation*}
\zeta_{f}(t)=\prod_{k=0}^{n}\mathrm{det}(H^k(\mathrm{Id})-tH^{k}(f))^{(-1)^{k+1}}.
\end{equation*}
\end{lemat}
Proof can be found in \cite{franks}. Note that all the Lefschetz numbers in the series vanish if and only if $\zeta_{f}(t)$ is equal to $1$ for all $t$. Hence, the map is LPPF if and only if $\zeta_{f}(t)$ is equal to $1$ for all $t$ and consequently if this condition does not hold then $f$ has a periodic point.\\
\indent The following theorem appeared in \cite{ls} for products of even dimensional spheres. Using the same method one can prove it for a wider class of spaces.
\begin{twierdzenie}\label{teo parz}
Let $X$ be a space with only even cohomology groups nonzero, $f:H^*(X,\mathbb{Q})\to H^*(X,\mathbb{Q})$ a ring homomorphism. Then there is some $n\in\mathbb{N}$ such that $L(f^n)\neq 0$. In particular a map between such spaces always has a periodic point.
\begin{proof} 
In our case
$$
\zeta_{f}(t)=\prod_{\substack{k\in\{0,\ldots,n\}\\2|k}}\mathrm{det}(H^k(\mathrm{Id})-tH^{k}(f))^{(-1)}.
$$
When $H^k(X,\mathbb{Q})$ is nonzero, the expression $\mathrm{det}(H^k(\mathrm{Id})-tH^{k}(f))$ is a polynomial of degree greater than one. So $\zeta_f(t)$ is different than $1$ for some $t$.
\end{proof}
\end{twierdzenie}

\section{main result}

Let $X=R\times E$ be a product of ENR's $R$ and $E$
and $f:X\to X$ a continuous map.\\
Let $F:H^*(X,\mathbb{Q})\to H^*(X,\mathbb{Q})$ be a ring homomorphism and $H^*(\pi_R):H^*(R,\mathbb{Q})\otimes H^*(E,\mathbb{Q})\to H^*(R,\mathbb{Q})$ be a map induced by a projection $\pi_R:R\times E\to R$. Define similarly $H^*(\pi_E):H^*(R,\mathbb{Q})\otimes H^*(E,\mathbb{Q})\to H^*(E,\mathbb{Q})$ induced by a projection $\pi_E:R\times E\to E$.\\
Denote by $F_{odd}$ a composition of $F$ with an embedding $H^*(R,\mathbb{Q})\ni \alpha\mapsto\alpha\otimes 1\in H^*(R,\mathbb{Q})\otimes H^0(E,\mathbb{Q})$ and by $F_{ev}$ a composition of $F$ with an embedding $H^*(E,\mathbb{Q})\ni \alpha\mapsto 1\otimes\alpha\in H^0(R,\mathbb{Q})\otimes H^*(E,\mathbb{Q})$.
\begin{lemat}
Let $F:H^*(X,\mathbb{Q})\to H^*(X,\mathbb{Q})$ be a ring homomorphism, $X=R\times E$ a product of ENR's. Then
$$
\mathrm{tr}(F)=\mathrm{tr}(H^*(\pi_R)F_{odd}\otimes H^*(\pi_E)F_{ev}).
$$

\begin{proof}
Using the K\"unneth Theorem we get that $H^*(X,\mathbb{Q})\simeq H^*(R,\mathbb{Q})\otimes H^*(E,\mathbb{Q})$.
Denote by $\alpha_0^{d(\alpha_0)},\ldots,\alpha_s^{d(\alpha_s)}$ a basis of $H^*(R,\mathbb{Q})$ as graded vector space, where $d(\alpha_i)$ denotes degree of $\alpha_i$ as a cohomology class. Similarly denote by $\beta_0^{d(\beta_0)},\ldots,\beta_t^{d(\beta_t)}$ a basis of $H^*(E,\mathbb{Q})$. We arrange the basis so that the sequences $d(\alpha_k)$ and $d(\beta_l)$ are nondecreasing. By $\widehat{\alpha_k}^{d(\alpha_k)}$ and $\widehat{\beta_l}^{d(\beta_l)}$ we will understand $H^*(\pi_R)(\alpha_k^{d(\alpha_k)})$ and $H^*(\pi_E)(\beta_l^{d(\beta_l)})$. We want to compute the coefficients of $F(\alpha_i^{d(\alpha_i)}\otimes\beta_{j}^{d(\beta_{j})})$. Let $\eta:=d(\alpha_i)$, $\mu:=d(\beta_j)$.
\begin{align*}
F(\alpha_i^{\eta}\otimes\beta_{j}^{\mu})
&=F(\widehat{\alpha_i}^{\eta}\smile\widehat{\beta_{j}}^{\mu})\\
&=F(\widehat{\alpha_i}^{\eta})\smile F(\widehat{\beta_{j}}^{\mu})\\
&=\sum_{d(\alpha_{k})+d(\beta_{s})=\eta}a_{k,s}~\widehat{\alpha_{k}}^{d(\alpha_{k})}\smile\widehat{\beta_{s}}^{d(\beta_{s})}\smile\sum_{d(\alpha_{m})+d(\beta_{l})=\mu}b_{m,l}~\widehat{\alpha_{m}}^{d(\alpha_{m})}\smile\widehat{\beta_{l}}^{d(\beta_{l})}\\
&=\mathcal{A}+\sum_{k} a_{k,0}\widehat{\alpha_{k}}^{\eta}\smile\sum_{p}b_{0,p}\widehat{\beta_{p}}^{\mu}\\
&=\mathcal{A}+\sum_{k,p} a_{k,0}b_{0,p}\widehat{\alpha_{k}}^{\eta}\smile \widehat{\beta_{p}}^{\mu},
\end{align*}
where $\mathcal{A}$ is a sum of all the elements of the form $\widehat{\alpha_{k}}^{d(\alpha_{k})}\smile\widehat{\beta_{s}}^{d(\beta_{s})}$ for $d(\alpha_{k})$, $d(\beta_{s})\neq 0$ with proper coefficients.
When we calculate a trace of $F$ we are interested only in the coefficient $a_{i,0}b_{0,j}$ which appears in the final line of the computation. It is the same number as corresponding coefficient for $\widetilde{\pi}_{H^*(R,\mathbb Q)}F_{odd}\otimes\widetilde{\pi}_{H^*(E,\mathbb Q)}F_{ev}$.
\end{proof}
\end{lemat}
Note, that
$$
\mathcal{L}(f^{k})=\sum_{n=0}^{\infty}(-1)^n\mathrm{tr}(H^n(f^{k}))=\sum_{n=0}^{\infty}\mathrm{tr}((-1)^n H^n(f^{k}))=
\mathrm{tr}\sum_{n=0}^{\infty}(-1)^n H^n(f^{k})=\mathrm{tr}T^k.
$$
where $T^k:H^*(X,\mathbb{Q})\to H^*(X,\mathbb{Q})$ is defined as $T^k:=\sum_{n=0}^{\infty}(-1)^n H^n(f^{k})$.
\begin{twierdzenie}\label{rown}
Let $X=R\times E$ a rational exterior space $R$ and a space $E$ with only even cohomology nonzero and let $f:X\to X$ be a continuous map. If $f$ is LPPF then there is an eigenvalue of $A^{odd}(f)$ equal to a root of unity. In particular $f$ has a periodic points if it does not have such an eigenvalue.
\end{twierdzenie}
\begin{proof}
From the above lemma we can conclude 
\begin{align*}
\mathcal{L}(f^k)
&=\mathrm{tr}(T^k)\\
&=\mathrm{tr}(H(\pi_R)T^k_{odd})\mathrm{tr}(H(\pi_E)T^k_{ev})\\
&=\mathrm{tr}(H(\pi_R)\sum_{n=0}^{\infty}(-1)^n H^n(f^{k})_{odd})\mathrm{tr}(H(\pi_E)\sum_{n=0}^{\infty}(-1)^n H^n(f^{k})_{ev})\\
&=\mathrm{tr}(\sum_{n=0}^{\infty}(-1)^n H^n(\pi_R)H^n(f^{k})_{odd})\mathrm{tr}(\sum_{n=0}^{\infty}(-1)^n H^n(\pi_E)H^n(f^{k})_{ev}).
\end{align*}
We get a product of Lefschetz numbers of two maps, the first is a ring endomorphism of cohomology ring of $R$, and the second a ring endomorphism of cohomology ring of $E$. The second map cannot always vanish as it is shown in Theorem \ref{teo parz} hence for all these numbers to be zero the first part must vanish at least for some $k$. All the eigenvalues of $\widetilde{\pi}_{H^*(R,\mathbb Q)}H^*(f^{k})_{odd}$ are equal to the eigenvalues of $f^{k}|_{H^*(R,\mathbb Q)\otimes H^0(E,\mathbb Q)}$. So by Theorem \ref{duan} there is a root of unity among the eigenvalues of $A^{odd}(f)$.
\end{proof}
\begin{obserwacja}
In the Theorem \ref{rown} we do not need $X$ to be a product, it is enought if $X$ is a space with rational cohomology ring isomorphic to a tensor product of an exterior algebra on odd degree generators and an algebra with only even degree elements.
\end{obserwacja}

We note that for a slightly better behaved class of spaces the criterion can be strengthened.
\begin{definicja} Let $X$ be as in the above observation. Let us assume that the algebra with even degree elements has at most one algebra generator in each degree. Moreover, each such generator squares to zero and the product of all these generators is the generator of the highest nontrivial degree. We call such a space an extended rational exterior space.
\end{definicja}
\begin{przyklad}
A product of a rational exterior space with $\mathbb{S}^{k_1}\times ...\times \mathbb{S}^{k_l}$ for pairwise different even natural numbers $k_1,...,k_l$ forms such a space.
\end{przyklad}
Following the exposition in \cite{duan} (with appropriate modifications) we choose a base $A=\{x_1,...,x_k,y_1,...,y_l\}$ of $A(X)$ consisting of homogeneous elements where $x_i$ denote elements of odd degree and $y_i$ denote elements of even degree. Under our assumptions this gives a base $H$ of the rational cohomology ring $H^*(X,\mathbb Q)$ consisting of $1$ and the cup products of the above generators. Consider the odd length function $\tilde{l}$ on elements of $H$ defined by $\tilde{l}(1)=1$ and $\tilde{l}(x_{i_1}\smile...\smile x_{i_r}\smile y_{j_1}\smile...\smile y_{j_q})=r$ as well as the standard length operator $l$ defined by $l(1)=1$ and $l(x_{i_1}\smile...\smile x_{i_r}\smile y_{j_1}\smile...\smile y_{j_q})=r+q$.
\begin{obserwacja} For $\alpha\in H$ we have $(-1)^{\mathrm{deg}\alpha}=(-1)^{\tilde{l}(\alpha)}$.
\end{obserwacja}
Consider next the duality operator
$$d(x_{i_1}\smile...\smile x_{i_r}\smile y_{j_1}\smile...\smile y_{j_q}):=x_{i'_1}\smile...\smile x_{i'_m}\smile y_{j'_1}\smile...\smile y_{j'_n},$$
where $i'$ and $j'$ denote the complement of $i$ and $j$ (preserving the order). Let $s:H\rightarrow \{0,1\}$ be the sign operator defined by $s(x_{i_1}\smile...\smile x_{i_r}\smile y_{j_1}\smile...\smile y_{j_q})=0$ when the sign of the permutation $(i_1,...,i_r,i'_1,...,i'_m)$ is even and $1$ otherwise ($j$ and $j'$ are omitted here since they correspond to the even degree generators which commute with any other element in the cohomology ring).
\begin{obserwacja} Given $(\alpha,\beta)\in H\times H$ with $l(\alpha)\geq l(d(\beta))$ we have
$$
\alpha\smile\beta = \left\{ \begin{array}{ll}
(-1)^{s(\alpha)} x_1\smile ...\smile x_k\smile y_1\smile ...\smile y_l, & \textrm{if $\beta=d(\alpha)$},\\
0, & \textrm{otherwise.}
\end{array} \right.
$$
\end{obserwacja}
Let $X$ be an extended rational exterior space. For a map $f:X\rightarrow X$ and $\alpha\in H$ we have
$$H^{*}(f)(\alpha)=\lambda_{\alpha}\alpha+\sum\limits_{\substack{\beta\in H\setminus\{\alpha\} \\ l(\alpha)\leq l(\beta)}} \sigma_{\alpha,\beta}\beta.$$
Due to the previous observation
$$H^{*}(f)(\alpha)\smile d(\alpha)=(-1)^{s(\alpha)}\lambda_{\alpha}x_1\smile ...\smile x_k\smile y_1\smile ...\smile y_l.$$
Note that the definition of the Lefschetz number $L(f)$ implies
$$L(f)=\sum\limits_{\alpha\in H}(-1)^{\mathrm{deg}\alpha}\lambda_{\alpha}=\sum\limits_{\alpha\in H}(-1)^{\tilde{l}(\alpha)}\lambda_{\alpha}.$$
Which together with the preceding discussion gives us
$$L(f)x_1\smile ...\smile x_k\smile y_1\smile ...\smile y_l=\sum\limits_{\alpha\in H}(-1)^{\tilde{l}(\alpha)+s(\alpha)}H^{*}(f)(\alpha)\smile d(\alpha).$$
The right hand side of this equality can also be given by
$$(x_1-H^{*}(f)(x_1))\smile ...\smile (x_k-H^{*}(f)(x_k))\smile (y_1+H^{*}(f)(y_1))\smile ...\smile (y_l+H^{*}(f)(y_l)).$$
Hence, we arrive at the following lemma:
\begin{lemat} Let $X$ be an extended rational exterior space. If we treat $A(X)$ as a subspace of $H^*(X,\mathbb{Q})$ (this can be done due to the chosen basis $A$ and $H$) we have the formula
$$
L(f)x_1\smile ...\smile x_k\smile y_1\smile ...\smile y_l=$$
$$
(x_1-A(f)(x_1))\smile ...\smile (x_k-A(f)(x_k))\smile (y_1+A(f)(y_1))\smile ...\smile (y_l+A(f)(y_l)).
$$
\begin{proof} Let $H(r)$ denote the subspace of $H^*(X,\mathbb{Q})$ generated by those $\alpha\in H$ with $l(\alpha)\geq r$. Note that $H(r)=0$ when $r>k+l$. Note also that the cup product gives a map $\smile : H(r)\times H(p)\rightarrow H(r+p)$. By the definition of $A$ we have $f_*(x_i)-A(f)(x_i)\in H(2)$. This implies that the difference
$$(x_1-v(x_1))\smile ...\smile (x_k-H^{*}(f)(x_k))\smile (y_1+H^{*}(f)(y_1))\smile ...\smile (y_l+H^{*}(f)(y_l))-$$
$$
(x_1-A(f)(x_1))\smile...\smile (x_k-A(f)(x_k))\smile (y_1+A(f)(y_1))\smile ...\smile (y_l+A(f)(y_l))
$$
lies in $H(k+l+1)=0$. Now the desired equality follows from the preceding discussion.
\end{proof}
\end{lemat}
Now we are ready to prove the strengthening of our criterion for extended rational exterior spaces
\begin{twierdzenie} Let $X$ be an extended rational exterior space and let $f:X\rightarrow X$ be a continuous map. Then $L(f^n)=\Pi^r_{i=1}(1-\lambda_i^n)\Pi^k_{i=1}(1+\sigma_l^n)$ where $\lambda_i$ are the eigenvalues of $A^{odd}(f)$ and $\sigma_i$ are the eigenvalues of $A^{ev}(f)$. Hence, $f$ is Lefschetz periodic point free if and only if one of the following conditions hold
\begin{enumerate}
\item there exists an integer $i$ such that $\lambda_i=1$,
\item there exist integers $i$ and $j$ such that $\lambda_i=\sigma_j=-1$.
\end{enumerate}
In, particular if neither of the conditions above is satisfied the $f$ has a periodic point.
\begin{proof} By standard linear algebra and the preceding discussion we have
$$
L(f)x_1\smile ...\smile x_k\smile y_1\smile ...\smile y_l=$$
$$\mathrm{det}(\mathrm{Id}-A(f))x_1\smile ...\smile x_k\smile (y_1+A(f)(y_1))\smile ...\smile (y_l+A(f)(y_l))$$
However, since $y_i$ are the only generators in each degree of $A^{ev}(X)$ we have that $A(f)(y_i)=\sigma_i y_i$. With this we can change the right hand side to
$$\mathrm{det}(\mathrm{Id}-A^{odd}(f))(1+\sigma_1)...(1+\sigma_l)x_1\smile ...\smile x_k\smile y_1\smile ...\smile y_l.$$
hence after substituting $f^n$ for $f$ we arrive at the formula:
$$L(f^n)=\Pi^r_{i=1}(1-\lambda_i^n)\Pi^k_{i=1}(1+\sigma_l^n).$$
To see the further claim note that $L(f)$ vanishes if and only if there exists $\lambda_i$ which is equal to $1$ or $\sigma_i$ which is equal to $-1$. If $1$ is among the eigenvalues of $A^{odd}(f)$ then all the numbers $L(f^n)$ vanish. On the other hand if $1$ is not among these eigenvalues but $-1$ is among the eigenvalues of $A^{ev}(f)$ then it is necessary for $L(f^2)$ to vanish that there is a $-1$ among the eigenvalues od $A^{odd}(f)$ since $i$ cannot be among the values of $A^{ev}(f)$ since this is a diagonal matrix on a rational vector space. One readily checks that in this case all of the $L(f^n)$ vanish. The final claim is automatic from the Lefschetz Fixed Point Theorem.
\end{proof}
\end{twierdzenie}

We finish this section by pinpointing a large class of sphere bundles to which our results apply.
\begin{propozycja}\label{prop z ciagiem dokl}
Let $B$ be a simply connected manifold of dimension $n$ with rational cohomology ring isomorphic to a tensor product of an exterior algebra on odd degree generators and an algebra with only even degree elements and let $k\in\mathbb{N}$ be odd integer $k\geq n$. Then any $\mathbb{S}^k$-bundle $E$ over $B$ has rational cohomology isomorphic to a tensor product of an exterior algebra on odd degree generators and an algebra with only even degree elements.
\begin{proof} It is well known that the second page of the Serre spectral sequence in cohomology (see e.g. chapter 5 in \cite{ah}) is of the form
$$E^{p,q}_2=H^{p}(B,H^{q}(\mathbb{S}^k,\mathbb{Q})).$$
From this we conclude that the only non-vanishing rows of the second page are the 0th and k-th row. So only the boundary operator on page $k+1$ can be non-zero. But on that page the derivative jumps $(k+1)$ steps to the right and since $k+1>n$ this means that it hits a column of zeroes. This shows that the cohomology of $E$ (which are isomorphic to the final page of $E^{p,q}_r$) are isomorphic as $\mathbb{Q}$ vector spaces to the cohomology of $B\times\mathbb{S}^k$. We choose a class $\alpha\in H^k(E,\mathbb{Q})$ corresponding to a non-zero class in $H^0(B,H^{k}(\mathbb{S}^k,\mathbb{Q}))$. Then the ring structure is defined by the following properties
\begin{enumerate}
\item $\alpha^2=0$ since $k$ is odd.
\item The cup product on classes corresponding to classes in $H^{p}(B,H^{0}(\mathbb{S}^k,\mathbb{Q}))$ is prescribed by the mapping induced by the projection map $\pi:E\rightarrow B$.
\item Multiplying classes corresponding to classes in $H^{p}(B,H^{0}(\mathbb{S}^k,\mathbb{Q}))$ by $\alpha$ is prescribed by the multiplicative structure of the spectral sequence.
\end{enumerate}
This implies that the ring structure is also as in the tensor product (where we add the generator $\alpha$ in the k-th degree to the exterior algebra).
\end{proof}
\end{propozycja}
\begin{obserwacja} Note that this proof does not hold in general when $k$ is even since then $\alpha^2$ might not be zero as is the case for example in the twisted bundle of $\mathbb{S}^2$ over $\mathbb{S}^2$.
\end{obserwacja}
Combining this proposition with our main result we get the following corollary.
\begin{wniosek} Take $E$ as above and let $f:E\rightarrow E$ be a continuous map. If there are no roots of unity among the eigenvalues of $A^{odd}(f)$ then $f$ has periodic points.   
\end{wniosek}

\section{counterexamples}
In this section we give two counterexamples to the claim in \cite{duan}, that a tower of odd
dimensional sphere bundles is a rational exterior space. \\
\indent The first example is the Kodaira-Thurston manifold, which is a bundle of $\mathbb{S}^1$ over $\mathbb{T}^3$. This space was discussed in the paper of William Thurston, see \cite{th}.\\
This manifold is obtained by identifying the points of $\mathbb{T}^2\times\mathbb{R}\times\mathbb{S}^1$ in the following way
$$
(x,y,t,z)\sim(x,y,t+j,z+jx),~j\in\mathbb Z.
$$
One can easily see that this is equivalent to taking $\mathbb{T}^2\times I\times\mathbb{S}^1$ with the boundaries identified by the function:
$$
f(x,y,0,z)=(x,y,1,z+x).
$$
From this we can conclude that the projection onto the first three coordinates gives this manifold the stucture of a circle bundle over $\mathbb{T}^3$. The Kodaira-Thurston manifold (denoted $X$) has a trivialization of the cotangent bundle given by the forms: $dx, dy, dt, dz-tdx.$ 
Computing the de Rham cohomology of $X$ we get 
$$H^0(X,\mathbb{Q})=\mathbb{Q},
H^1(X,\mathbb{Q})=\mathbb{Q}^3,
H^2(X,\mathbb{Q})=\mathbb{Q}^4,
H^3(X,\mathbb{Q})=\mathbb{Q}^3,
H^4(X,\mathbb{Q})=\mathbb{Q}.$$
We can see that $X$ cannot be a rational exterior space, because it has inappropriate dimensions of cohomology groups.\\
\indent Now we will give an example of a $5$-sphere bundle over a simply connected space $\mathbb{S}^3\times\mathbb{S}^3$, which is not a rational exterior space and hence gives another counterexample to the claim from \cite{duan} (since it is a $5$-sphere bundle over the total space of a $3$-sphere bundle over a $3$-sphere). Let us consider a map $f:\mathbb{S}^3\times\mathbb{S}^3\to\mathbb{S}^6$ shrinking 3-cells to a point. It induces an isomorphism in the sixth cohomology. Let us pull back the bundle $T\mathbb{S}^6$ through this map. The unit sphere bundle $X$ of the vector bundle $f^*(T\mathbb{S}^6)$ turns out to be the desired bundle. To see this we compute the fifth and sixth cohomology groups of $X$ via the Gysin exact sequence with rational coefficients (see e.g. \cite{chern}):
\begin{equation*}
0\to H^5(X,\mathbb{Q})\to H^{0}(\mathbb{S}^3\times\mathbb{S}^3,\mathbb{Q})\xrightarrow{e\smile} H^{6}(\mathbb{S}^3\times\mathbb{S}^3,\mathbb{Q})\to H^6(X,\mathbb{Q})\to 0.
\end{equation*}
Where the zeroes on the left and right correspond to $H^{5}(\mathbb{S}^3\times\mathbb{S}^3,\mathbb{Q})$ and $H^{1}(\mathbb{S}^3\times\mathbb{S}^3,\mathbb{Q})$ respectively and $e$ is the euler class of $f^*(T\mathbb{S}^6)$. Due to the naturality of the euler class and $f$ being an isomorphism in the sixth cohomology the Euler class is easily computed to be equal to $2$ (since the euler characteristic of $\mathbb{S}^6$ is equal to $2$). This implies that the middle arrow is an isomorphism and so:
\begin{equation*}
H^5(X,\mathbb{Q})\cong H^6(X,\mathbb{Q})\cong 0.
\end{equation*}
Now using the Serre spectral sequence one can easily compute the rest of the cohomology groups and the only non-zero ones are:
\begin{eqnarray*}
H^0(X,\mathbb{Q})\cong H^{11}(X,\mathbb{Q})\cong \mathbb{Q},\\
H^3(X,\mathbb{Q})\cong H^8(X,\mathbb{Q})\cong \mathbb{Q}\oplus\mathbb{Q}.
\end{eqnarray*}
Hence it cannot be a rational exterior space (for example the cup product of any two elements of degree $3$ is zero which wouldn't be possible for a rational exterior space with third cohomology of dimesnion $2$).

\end{document}